\documentclass[a4paper, twoside]{article}

\usepackage[british,UKenglish,USenglish,american]{babel}
\usepackage{graphicx}           

\usepackage{authblk}
\usepackage{fancyhdr}
\usepackage{rotating}
\usepackage{amsmath}
\usepackage{amssymb}
\usepackage{amsthm}
\usepackage{tikz}
\usepackage{url}
\usepackage{cite}
\usepackage{enumerate}
\usepackage{mathtools}
\usepackage{setspace}
\usepackage{color}
\usepackage[normalem]{ulem}
\usetikzlibrary{positioning}

\newtheorem{definition}{Definition}

\newtheorem{theorem}{Theorem}

\newtheorem{conjecture}{Conjecture}
\newtheorem{fact}{Fact}

\newcommand{\F}{\mathbb{F}}

\begin{document}

\title{Definably simple stable groups with finitary groups of automorphisms}

\author{Ulla Karhum\"{a}ki}

\affil{School of Mathematics, University of Manchester, UK\\
\texttt{ulla.karhumaki@manchester.ac.uk }}

\pagestyle{fancy}
\fancyhead{} 
\fancyhead[RO,LE]{\thepage}
\fancyhead[LO]{DEFINABLY SIMPLE STABLE GROUPS}
\fancyhead[RE]{Ulla Karhum\"{a}ki}
\fancyfoot{}


\maketitle

\begin{abstract}

We prove that infinite definably simple locally finite groups of finite centraliser dimension are simple groups of Lie type over locally finite fields. Then, we identify conditions on automorphisms of a stable group that make it resemble the Frobenius maps, and allow us to classify definably simple stable groups in the specific case when they admit such automorphisms.

\end{abstract}

\section{Introduction}\label{section:introduction}

This paper combines two different trains of investigations; on the one hand we continue the work on locally finite groups following Thomas \cite{thomas21983} and on the other hand we classify definably simple stable groups which admit groups of (non-definable) automorphisms which resemble Frobenius maps, and, in particular, have finite groups of fixed points. We call such groups \emph{finitary groups of automorphisms}; the definitions are given in Sections \ref{sec:premlim} and~\ref{sec:finitary}.

We start with generalising the classical result by Thomas \cite{thomas21983} about infinite simple locally finite groups of finite centraliser dimension. We prove the following:

\begin{theorem} \label{th:definably-simple}
An infinite definably simple locally finite group of finite centraliser dimension is a simple group of Lie type over a locally finite field.
\end{theorem}

Proof of Theorem~\ref{th:definably-simple} uses the Classification of Finite Simple Groups (CFSG).

Then we study definably simple stable groups which admit a finitary group of automorphisms.

\begin{theorem}\label{stable-extension} Every infinite definably simple stable group admitting a finitary group of automorphisms is a Chevalley group over an algebraically closed field of positive characteristic.
\end{theorem}

Theorem~\ref{stable-extension} also depends on the CFSG.

In Section~\ref{section:fRM}, we show that in case of  groups of finite Morley rank, Theorem~\ref{stable-extension} can be proven without the use of the CFSG:

\begin{theorem}\label{theorem:fRM}
An infinite simple group of finite Morley rank admitting a finitary group of automorphisms is a Chevalley group over an algebraically closed field of positive characteristic.
\end{theorem}

This paper is organised in the following way. In Section~\ref{sec:premlim}, we give preliminaries and facts needed in the rest of the paper. Then, in Section~\ref{sec:locally-finite}, we prove Theorem~\ref{th:definably-simple}. In Section~\ref{sec:finitary}, we introduce a finitary group of automorphisms of an algebraic structure $\mathcal{M}$ and the corresponding locally finite elementary substructure $\mathcal{M}_*$ of $\mathcal{M}$. Then, in Section~\ref{sec:stable}, we prove Theorem~\ref{stable-extension}. Finally, in Section~\ref{section:fRM}, we first introduce groups of finite Morley rank together with some background results, and then prove Theorem~\ref{theorem:fRM} without the use of the CFSG. In particular, this confirms the Cherlin-Zilber Conjecture in the specific case of a simple group of finite Morley rank admitting a finitary group of automorphisms.

\section{Preliminaries and background results}\label{sec:premlim}

In this section we give definitions and facts needed in the rest of the paper.

A group $G$ is called \emph{locally finite} if every finitely generated subgroup of $G$ is finite. For details on locally finite groups the reader may consult \cite{KegelOtto,Hartley-Borovik}. A group $G$ is called \emph{definably simple} if it has no proper non-trivial definable normal subgroups. If lengths of chains of centralisers in a group $G$ are bounded, the length of the longest such chain is called the \emph{centraliser dimension} of $G$.

\emph{Chevalley groups} are defined  by generators and relations and provide a description of important classes of simple algebraic groups (including those over algebraically closed fields) as \emph{abstract} groups, ignoring the underlying algebraic variety. In this paper we won't present the definition of Chevalley groups or twisted analogues of Chevalley groups but a reader unfamiliar with the construction may find all necessary details in \cite{Carter1971}.

The classification of groups of points of simple algebraic groups over algebraically closed fields amounts to the statement that they are isomorphic to Chevalley groups over corresponding fields; the reverse statement is also true.

In this paper, following the finite group terminology, we shall apply an umbrella term \emph{groups of Lie type} to all Chevalley groups, twisted analogues of Chevalley groups, and  non-algebraic twisted groups of Lie type $^2B_2$, $^2F_4$, or $^2G_2$ over finite, locally finite, and pseudofinite fields \cite{Carter1971,Steinberg}.

The following classical result is due to Simon Thomas.

\begin{fact}[Thomas \cite{thomas21983}]\label{thomas}\label{Fact:Thomas-locally-finite} An infinite simple locally finite group which satisfies the minimal condition on centralisers is a simple group of Lie type over a locally
finite field.

\end{fact}

The other important result by Thomas states that being a Chevalley group of fixed type over some field $K$ is a first-order axiomatisable property. This, together with results by Point and by Ryten, allows us to prove Theorem~\ref{extension2} in Section~\ref{sec:stable}.

\begin{fact}[Thomas \cite{thomas1983}, Theorem 18]\label{Fact:thomas-Chevalley} For each Chevalley group of fixed type $X$, the class $\{X(K)\mid K$ is a field $ \}$ is finitely axiomatisable.
\end{fact}

A theory $T$ is called \emph{pseudofinite} if
it has infinite models but every sentence in it has a finite model. If $G$ is a simple pseudofinite group, then $G$ is elementary equivalent to an ultraproduct  $\prod G(q)/\mathcal{U}$ of finite groups $G(q)$ of a fixed Lie type over finite fields $\F_q$ \cite{Wilson1995}. For details on ultraproducts reader may consult \cite{BellSlomson}.

\begin{fact}[Point \cite{Point}]\label{point} Let $\{X(F_i) \mid i \in I\}$ be a family of Chevalley or twisted Chevalley groups of fixed type $X$ over finite or pseudofinite fields $F_i$, and let
$\mathcal{U}$ be a non-principal ultrafilter on the set $I$. Then,
\[\prod_{i \in I}X(F_i)/\mathcal{U} \cong X\left(\prod_{i\in I} F_i/\mathcal{U}\right).
\]
\end{fact}

\begin{fact}[Ryten \cite{Ryten2007}, Chapter 5]\label{ryten}
Let $G(q)$ be finite groups of fixed Lie type over finite fields $\F(q)$. Then groups $G(q)$ are uniformly bi-interpretable over parameters with $\F_q$, or, in the case of groups $^2B_2$, $^2F_4$, or $^2G_2$, with structures $(\F_q ,\sigma)$, for an appropriate field automorphism $\sigma$.

\end{fact}

Next, we briefly introduce stable structures. Detailed discussion of stable groups can be found in \cite{Wagner1997, Poizat2001a}. More generally, for introduction on stability theory, we refer the reader to \cite{Pillay2013}.

\begin{definition} A $\mathcal{L}$-structure $\mathcal{M}$ is called stable if there is no $\mathcal{L}$-formula $\phi(\bar{x},\bar{y}) \in Th(\mathcal{M})$ and tuples $\bar{a}_i$, $\bar{b}_j$ with $i,j < \omega$ such that $\mathcal{M}\models \phi(\bar{a}_i,\bar{b}_j)$ if and only if $i\leq j$.

\end{definition}

The fact below is well-known,  see for example \cite[Lemma~8.4.2]{Ziegler-Tent}.

\begin{fact}
A structure interpretable in a stable structure is stable.
\end{fact}

The following results by Baldwin and Saxl and by Duret play a key role in Section~\ref{sec:stable}, where we prove Theorem~\ref{stable-extension}.

\begin{fact}[Baldwin and Saxl \cite{baldwinsaxl}]\label{Fact:BaldwinSaxl} Stable groups have finite centraliser dimension.
\end{fact}

\begin{fact}[Duret, \cite{Duret1980}, Corollaire 6.6]\label{duret} Infinite locally finite stable fields are algebraically closed.
\end{fact}

\section{Locally finite definably simple groups of finite centraliser dimension}\label{sec:locally-finite}

In this section we prove our result on infinite definably simple locally finite groups of finite centraliser dimension, Theorem~\ref{th:definably-simple}; it is a natural extension of Fact~\ref{thomas}.

\medskip
\noindent
\textbf{Theorem \ref{th:definably-simple}}\emph{ An infinite definably simple locally finite group of finite centraliser dimension is a simple  group of Lie type over a locally finite field.}
\medskip

Our proof uses the CFSG.

\subsection{Proof of Theorem~\ref{th:definably-simple}}

We start by citing a general structural result about locally finite groups of finite centraliser dimension \cite{borovik-karhumaki, Buturlakin2018}.

\begin{fact} [Borovik and Karhum\"{a}ki \!\!\cite{borovik-karhumaki}, Buturlakin \cite{Buturlakin2018}]\label{th:structure-lfg} Let $H$ be a locally finite group of finite centraliser dimension $c$. Then $H$ has a  normal series
\[
1 \unlhd S \unlhd L  \unlhd  H,
\]
where

\begin{itemize}

\item[{\rm (a)}] $S$ is solvable of solvability degree bounded by a function of $c$,

\item[{\rm (b)}] $\overline{L} =L/S$ is a direct product $\overline{L} = \overline{L}_1 \times \cdots \times \overline{L}_m$ of finitely many non-abelian simple groups,

\item[{\rm (c)}] each $\overline{L}_i$ is either finite, or a group of Lie type over a locally finite field,
\item[{\rm (d)}]  the factor group $H/L$ is finite.

\end{itemize}
\end{fact}

Now we start the proof of Theorem~\ref{th:definably-simple}. We work in notation of Fact~\ref{th:structure-lfg}. Recall that, in a group with descending chain condition for centralisers, centralisers are definable. 

If $S \neq 1$ then $H$ contains a non-trivial abelian normal subgroup $A$. But then $A \leqslant Z(C_H(A))$ and $Z=Z(C_H(A))$ is a non-trivial definable abelian normal subgroup of $H$.  Since $H$ is definably simple, $Z=H$ and $H$ is abelian; being definably simple locally finite group, it is forced to be cyclic of prime order, which contradicts the assumption that $H$ is infinite. Hence $S=1$.

Now, by Fact \ref{th:structure-lfg}, $H$ contains a normal subgroup $L$ of finite index such that $L = L_1 \times \cdots \times L_m$ is a direct product of finitely many non-abelian simple groups $L_i$ each of which is a group of Lie type  over a locally finite field.

Next we prove that $L=H$; to reach that point, it suffices to prove that $L$ is definable in $H$.

Recall that a simple group $R$ is called \emph{uniformly simple} if for any two non-trivial elements $g$ and $r$, the length of expression of $g$, as a finite product of conjugates of $r$ or the inverse of $r$, is uniformly bounded. Now, by \cite{Ellers1999}, each $L_i$, being simple groups of Lie type, are uniformly simple in $L$. Thus, every element in each $L_i$ is a product of a bounded number of $L_i$-conjugates of some non-trivial element $ x_i \in L_i$. Further, $H$-conjugates of $x_i$ belong to the normal subgroup $L$ of $H$, that is, $$L_i=(x_i^{\pm 1})^{ L_i}\cdots (x_i^{\pm 1})^{ L_i} \subseteq (x_i^{\pm 1})^{H}\cdots (x_i^{\pm 1})^{H} \subseteq \langle x_i^H \rangle \leqslant L.$$ Take the maximal length of the product $(x_i^{\pm 1})^{ L_i}\cdots (x_i^{\pm 1})^{ L_i}$ (and thus the maximal length of the product $(x_i^{\pm 1})^{H}\cdots (x_i^{\pm 1})^{H}$ as well) for each $i$.  Then, $$L=L_1 \times \cdots \times L_m \subseteq \left((x_1^{\pm 1})^{H}\cdots (x_1^{\pm 1})^{H} \right) \cdots \left((x_m^{\pm 1})^{H}\cdots (x_m^{\pm 1})^{H} \right) \subseteq L.$$ This proves that $L$ is a definable subgroup of $H$, and thus, $L=H$ since $H$ is definably simple.

We should mention that one can also find a first-order formula without parameters defining $L$ in $H$: groups $L_i$, being groups of Lie type over a locally finite field, are unions of ascending chains of finite simple groups of the same Lie type over finite fields of the same characteristic, say $p$. Their orders are well-known (see, for example, the table \cite[pp.~490]{gorenstein1968}) and contain divisors of the form $p^k\pm 1$ for $k$ growing without bounds as subgroups in the chains grow bigger. Now the classical theorem by Zsigmondy \cite{Zsigmondy1892} ensures that numbers of the form $p^k\pm 1$ with unbounded $k$ have arbitrarily large prime divisors, which means that each group $L_i$ contains elements of arbitrarily large prime orders. Let $\ell_i \in L_i$ be an element of a prime order $t$ chosen in the way that $t > |L:L_i|$. Notice that all elements of order $t$ in $L$ belong to $L_i$. Similarly as above, by \cite{Ellers1999}, every element in $L_i$ is a product of a bounded number of $L_i$-conjugates of $\ell_i$. That is, similarly as above, we may prove that $L$ is definable, without parameters, in $H$.

Now, $H = L = L_1 \times \cdots \times L_m$ is a direct product of finitely many non-abelian simple groups $L_i$ each of which is a group of Lie type over a locally finite field. Thus, to complete the proof of Theorem  \ref{th:definably-simple}, it is enough to prove that $m=1$.  Assume the contrary, let $m >1$. The group $C_H(C_H({L_1}))$ is a definable and normal subgroup of a definably simple group $H$. Clearly $C_H(L_1) \nleqslant C_H(C_H({L_1}))$ since $C_H(L_1)$ contains $L_2$. Hence $C_H(C_H({L_1}))$ is a proper non-trivial definable normal subgroup in a definably simple group $H$ -- a contradiction.  \hfill $\square$

\section{Finitary automorphism groups}
\label{sec:finitary}

In this section, we concentrate on the other line of thought of our paper. The key definitions of this part of the paper are given below in a general form.

\begin{definition}\label{def:finitary} Let $\mathcal{M}$ be an infinite algebraic structure with underlying set $M$. We say that an infinite group $A$ of automorphisms of $\mathcal{M}$ is \emph{finitary}, if
\begin{itemize}
 \item[{\rm (1)}] For every $\alpha \in A \smallsetminus \{1\}$, the substructure of fixed points $\mathop{\rm Fix}_\mathcal{M}(\alpha)$ is finite.

\item[{\rm (2)}] If $X$ is a non-empty definable subset in $M$ and $X$ is invariant under the action of a non-trivial automorphism $\alpha\in A$, then there exists an element $x \in X$ with a finite orbit $x^A$  {\rm(}equivalently,  the stabiliser\/ $\mathop{\rm Stab}_A(x)$ has a finite index in $A${\rm )}.
 \end{itemize}
 Under assumptions {\rm (1)} and {\rm (2)} above, we define
\begin{itemize}
\item[{\rm (3)}]  the \emph{locally finite core} $\mathcal{M}_*$ of $\mathcal{M}$ as a substructure of $\mathcal{M}$ with the underlying set
\[
M_* = \{\, m \in M \mid \text{ the orbit }\; m^A \text{ is finite.}\}.
\]
\end{itemize}
\end{definition}

Notice that  assumption (2)  of Definition \ref{def:finitary}, applied to $X=M$, ensures that $M_*$ is non-empty. Notice further that $M_*$ is infinite for otherwise $A$ would have a finite orbit $x^A$ on the definable $A$-invariant subset $X=M \smallsetminus M_*$ with $x$ not belonging to $M_*$ -- a contradiction.

An important example of finitary groups of automorphisms is provided by Frobenius maps. Let $G$ be a group of points over an algebraically closed field $K$ of characteristic $p > 0$ of a simple algebraic group defined over the field $\mathbb{F}_{p}$, and $\phi$ the automorphism of $G$ induced by the Frobenius map $x \mapsto x^{p}$ on $K$. Then it is well-known that the group $\langle \phi \rangle$ generated by $\phi$ is a finitary group of automorphisms.

Definition~\ref{def:finitary} together with the well-known Tarski-Vaught test \cite{Hodges1993} for elementary substructures allows us to prove that the locally finite core $\mathcal{M}_*$ is actually a locally finite elementary substructure of $\mathcal{M}$.

\begin{theorem}\label{elementary-substructure}
Let $\mathcal{M}$ be an algebraic structure and $A$ a finitary group of automorphisms of $\mathcal{M}$. Then, $\mathcal{M}_*$ is a locally finite elementary substructure of $\mathcal{M}$.
\end{theorem}
\begin{proof}

First we check that $M_*$ with operations and relations inherited from $\mathcal{M}$ is a locally finite substructure of $\mathcal{M}$.

Let $x_1,\dots,x_n \in M_*$, and $\mathcal{X}= \langle x_1, \ldots, x_n \rangle$ the substructure generated by elements $x_1,\dots, x_n$ in $M$, that is, the minimal subset of $M$ which contains $x_1,\dots, x_n$ and is closed under all algebraic operations from $\mathcal{M}$. Then the stabilisers of $x_1,\dots,x_n$ in $A$ have finite indices in $A$ and their intersection
\[
B = {\rm Stab}_A(x_1) \cap \cdots \cap {\rm Stab}_A(x_n)
\]
also has a finite index in $A$, hence is not trivial. Here $B$ fixes every element in $X$, hence $A$-orbits of elements in $X$ are finite and $X\subseteq M_*$, which also proves that $M_*$ is closed under all operations from $\mathcal{M}$ and is therefore a substructure of $\mathcal{M}$. Moreover, $X \subseteq \mathop{\rm Fix}_M(B)$ is finite by the definition of finitary groups of automorphisms, hence $\mathcal{M}_*$ is locally finite.

Then, we shall show that every $\mathcal{M}_*$-definable non-empty subset $X\subseteq M$ has a point in $M_*$. This will prove the theorem by invoking the Tarski-Vaught Test.

Indeed since $X$ is a $M_*$-definable set it is of the form
\[
X=\{x \in M \mid \varphi(x, \overline{m})\},
\]
where $\overline{m}=(m_1,\dots,m_k) \in M_*$. Every $m_i$ has a finite $A$-orbit and therefore its stabiliser $S_i = {\rm Stab}_A(m_i)$ has finite index in $A$, hence the intersection
\[
S = {\rm Stab}_A(m_1)\cap \cdots \cap {\rm Stab}_A(m_k)
\]
of all $S_i$'s also has a finite index in $A$. Clearly, for all $i = 1,\ldots, k$, $S$ fixes all $m_i$'s. Now the set $X$ is $S$-invariant and $S$ has a finite orbit $O$ in $X$ such that the point-wise stabiliser ${\rm Stab}_A(O)$ has a finite index in $A$; hence $A$ also has a
finite orbit in $X$. Elements in this orbit $O$ belong to $M_*$ so $X \cap M_* \neq \emptyset$ as is claimed.

\end{proof}

Later, in Sections~\ref{sec:stable} and~\ref{section:fRM} , when applying Definition \ref{def:finitary} to groups, we shall follow standard notation and refer to stabilisers of points as \emph{centralisers} and write $C_A(x)$ instead of ${\rm Stab}_A(x)$ and  $C_X(\alpha)$ instead of $\mathop{\rm Fix}_X(\alpha)$, etc.

\section{Stable groups admitting a finitary group of automorphisms}\label{sec:stable}

In this section we combine our two lines of investigation. As a result we give a classification of definably simple stable groups in the specific case in which these groups admit finitary groups of automorphisms.

\medskip
\noindent
\textbf{Theorem \ref{stable-extension}}\emph{ Every infinite definably simple stable group admitting a finitary group of automorphisms is a Chevalley group over an algebraically closed field of positive characteristic.}
\medskip

Note that this theorem implies that, in presence of  a finitary  group of automorphisms, a definably simple stable group  is simple.

Our proof uses the CFSG.
\subsection{Proof of Theorem~\ref{stable-extension}}

Proof of Theorem~\ref{stable-extension} follows as we combine together Theorems~\ref{th:definably-simple} and~\ref{elementary-substructure}. Indeed, we may first observe that the following fact is a corollary of our earlier results:

\begin{theorem}\label{extension2}
Every infinite definably simple group $G$ of finite centraliser dimension admitting a finitary group of automorphisms is a simple group of Lie type over a field of positive characteristic.
\end{theorem}

\begin{proof}
Let $G$ be an infinite definably simple  group of finite centraliser dimension admitting a finitary group of automorphisms $A$. Then the locally finite group $G_* \preceq G$ also has finite centraliser dimension and is definably simple.  By Theorem \ref{th:definably-simple}, $G_*$ is a  group of Lie type
over a locally finite field. Since $G$ is elementary equivalent to $G_*$, it is also a  group of Lie type over some field of positive characteristic in view of Facts~\ref{Fact:thomas-Chevalley},~\ref{point} and~\ref{ryten}.
\end{proof}

Theorem \ref{extension2} almost immediately leads to our result about stable groups.

Indeed, stable groups have finite centraliser dimension by Fact~\ref{Fact:BaldwinSaxl}; note further that  a field appearing in the statement of Theorem \ref{extension2} is interpretable in the group $G$ due to Facts~\ref{point} and~\ref{ryten}, and therefore Theorem \ref{stable-extension} follows from Theorem \ref{extension2} and Fact~\ref{duret}: Infinite locally finite stable fields are algebraically closed. We need only to recall that groups of Lie type over algebraically closed fields are Chevalley groups. This follows since algebraically closed locally finite fields are algebraic closures of finite prime fields but for any prime $p > 0$ there are no twisted Chevalley groups with underlying field $\bar{\mathbb{F}}_p$. This completes the proof of Theorem~\ref{stable-extension}.
\qed

It is worth mentioning that there exist definably simple stable groups which are not simple. The first example was given by Sela in \cite{Sela2013}, where he proved that the elementary theory of non-abelian free groups, $T_{fg}$, is stable. This provided an example since it was proven by several authors that a free group $F_n \models T_{fg}$ is definably simple (but not simple):

\begin{fact}[Bestiva and Feighn \cite{Feighn2010}; Kharlampovich and Myasnikov \cite{Myasnikov2013}; Perin, Pillay, Sklinos and Tent \cite{Perin2013}]\label{fact:defsimple} Any proper definable subgroup of a non-abelian free group $F_n$ is cyclic.
\end{fact}

\section{Groups of finite Morley rank admitting a finitary group of automorphisms}\label{section:fRM}

In this section we prove the analogue of Theorem~\ref{stable-extension} without the use of the CFSG in the specific case of groups of finite Morley rank.

\subsection{Groups of finite Morley rank and the Cherlin-Zilber Conjecture}

We consider groups of finite Morley rank via a \emph{ranked definable universe} following the notations of \cite{ABC, Bor2}.  This way we define a \emph{definable universe} $\mathcal{U}$ to be a set of sets closed under taking singletons (one-element sets), Boolean operations, direct products, projections and quotients (by equivalence relations in the universe). Further, we call definable universe $\mathcal{U}$ \emph{ranked} if it carries a rank function as defined by Borovik and Poizat \cite{Poizat2001a}. This means that we can assign, to every set $S  \in \mathcal{U}$, a natural number which behaves as a ``dimension'' of the corresponding set. A group $G$ is \emph{ranked} if its base set and graphs of multiplication and inversion belong to a ranked universe. It is well-known \cite{Bor2, Poizat2001a} that a simple ranked group is $\omega$-stable of finite Morley rank.

The motivation for study of infinite simple groups of finite Morley rank arises from the famous Cherlin-Zilber Conjecture made in the 1970's \cite{CHERLIN19791,Zilber1977}:

\begin{conjecture}[Cherlin-Zilber Conjecture]Simple infinite groups of finite Morley rank are algebraic groups over algebraically closed fields.
\end{conjecture}

A detailed discussion of the Cherlin-Zilber Conjecture, as well as development of relevant techniques and purely algebraic axiomatisation of the concept of (finite) Morley rank can be found in \cite{Bor2} and \cite{ABC}. In what follows, we prove that the Cherlin-Zilber Conjecture holds in the specific case when a simple group of finite Morley rank admits a finitary group of automorphisms, that is, we prove the following theorem (without the use of the CFSG).

\medskip
\noindent
\textbf{Theorem 3.} \emph{An infinite simple group of finite Morley rank admitting a finitary group of automorphisms is a Chevalley group over an algebraically closed field of positive characteristic.}
\medskip

\subsubsection{Background results}

We start by listing few properties of groups of finite Morley rank that we use in our analysis.

The following result states that for groups the property of being of finite Morley rank is preserved under elementary equivalence. It was proven by Poizat; see the introduction of \cite{Poizat2001a}.

\begin{fact}[Poizat \cite{Poizat2001a}]\label{defmorleyrank} Being a ranked group is a first-order axiomatisable property.
\end{fact}

The following crucial fact for groups of finite Morley rank follows from Zilber's Indecomposability Theorem \cite{Zilber1977}.

\begin{fact}[\cite{Marker2002}, Corollary 7.3.7]\label{defsimple}
A non-abelian simple group of finite Morley rank is definably simple.
\end{fact}

Finally, we comment the classification result for locally finite groups of finite Morley rank. It was first proven by Thomas \cite{thomas21983} who used the CFSG. Later results by Borovik \cite{MR1362813}, and by Alt{\i}nel, Borovik and Cherlin \cite{ABC} removed the reference to the CFSG, thus the following fact does not use the CFSG.

\begin{fact} \label{fact:locally-finite} A locally finite simple group of finite Morley rank  is a Chevalley group over an algebraically closed field of positive characteristic.
\end{fact}

\subsubsection{Proof of Theorem~\ref{theorem:fRM}}

Here, $G$ refers to an infinite simple group of finite Morley rank with a finitary group of automorphisms $A$. In view of Theorem \ref{elementary-substructure}, $G_* \preceq G$ and is therefore also a group of finite Morley rank by Fact~\ref{defmorleyrank}.

Simplicity of $G_*$ follows immediately from properties of groups of finite Morley rank we presented in the previous subsection. Since $G_*$ is elementary equivalent to $G$ it inherits the property of being definably simple. But then, by Fact~\ref{defsimple}, $G_*$ is indeed simple. Moreover, it immediate follows from Fact~\ref{fact:locally-finite} that $G_*$ is a Chevalley group over algebraically closed field of positive characteristic.
Theorem~\ref{theorem:fRM} follows now from the elementary embedding $G_* \preceq G$ together with the Fact~\ref{Fact:thomas-Chevalley}.
\qed \\

We should mention that Theorem~\ref{theorem:fRM} can be also proven using purely group theoretic methods using the following strategy. First one can shown that $G_*$ is a simple  locally finite group of finite Morley rank, thus a Chevalley group over an algebraically closed field \cite{ ABC, MR1362813}. Then it is possible to expand a specific configuration of the Curtis-Phan-Tits-Lyons Theorem \cite[Theorem~1.24.2]{GLS} present in $G_*$, from $G_*$ to $G$, and finally apply the Curtis-Phan-Tits-Lyons Theorem to $G$. The model theoretic approach we gave here is shorter and more transparent.

\subsection*{Acknowledgements}
The author is thankful to Alexandre Borovik for proposing the topic, for numerous helpful discussions and for precious suggestions on how to improve preliminary versions, and to Adrien Deloro for noticing a gap in the original proof of an earlier and weaker version of Theorem \ref{stable-extension} and for valuable comments and advice. Author also wants to thank the referee for valuable comments and suggestions to improve the paper.

The author is funded by the Vilho, Yrj\"{o} and Kalle V\"{a}is\"{a}l\"{a} Foundation of the Finnish Academy of Science and Letters.

\addcontentsline{toc}{section}{References}    
\bibliographystyle{plain}
\bibliography{Ulla_5March19}

\end{document}